\documentclass[11pt,a4paper]{amsart}

\usepackage{amssymb}
\usepackage{amsmath}
\usepackage{amsthm}
\usepackage[mathscr]{euscript}
\usepackage{verbatim}   
\usepackage{graphicx}   
\usepackage{enumitem}
\usepackage{color}     

\theoremstyle{plain}
\numberwithin{equation}{section}

\newtheorem{theorem}{Theorem}[section]
\newtheorem{lemma}[theorem]{Lemma}
\newtheorem{proposition}[theorem]{Proposition}
\newtheorem{corollary}[theorem]{Corollary}

\newtheorem{maintheorem}{Main Theorem}

\theoremstyle{definition}

\newtheorem{definition}[theorem]{Definition}

\newcommand\C{\mathbb{C}}

\newcommand\Z{\mathbb{Z}}

\newcommand\D{\mathbb{D}}

\newcommand{\N}{{\mathcal N}}
\newcommand{\Newt}{{\textrm{\bf Newt}}}
\newcommand{\AENG}{{\textrm{\bf NGraph}}}

    \usepackage{enumitem}

\theoremstyle{remark}
\newtheorem{remark}[theorem]{Remark}

\def\theta{\vartheta}

\def\S{\mathbb{S}}

\def\D{\mathbb{D}}

\def\ol{\overline}

\begin{document}

\title{A classification of postcritically finite Newton maps}

\begin{abstract}
The dynamical classification of rational maps is a central concern of holomorphic dynamics. Much progress has been made, especially on the classification of polynomials and some approachable one-parameter families of rational maps; the goal of finding a classification of general rational maps is so far elusive.  Newton maps (rational maps that arise when applying Newton's method to a polynomial) form a most natural family to be studied from the dynamical perspective.  Using Thurston's characterization and rigidity theorem, a complete combinatorial classification of postcritically finite Newton maps is given in terms of a finite connected graph satisfying certain explicit conditions.
\end{abstract}

\begin{author}[Lodge]{Russell Lodge}
\email{russell.lodge@indstate.edu}
\address{Department of Mathematics and Computer Science, Indiana State University, 200 North Seventh Street, Terre Haute, IN 47809, USA}
\end{author}

\begin{author}[Mikulich]{Yauhen Mikulich}
\email{y.mikulich@gmail.com}
\address{All\'ee Leotherius 2, 1196 Gland, Switzerland}
\end{author}

\author[Schleicher]{Dierk Schleicher}
\email{dierk.schleicher@univ-amu.fr}
\address{Aix--Marseille Universit\'e and CNRS, Institut de Math\'ematiques de Marseille, UMR 7373, 163 Avenue de Luminy, 13009 Marseille, France}
\keywords{Newton map; rational map; parameter space; renormalization; Hubbard tree; combinatorial classification; extended Newton graph; Thurston's theorem.}

\thanks{This research was partially supported by the Deutsche Forschungsgemeinschaft (DFG), as well as the advanced grant 695621 HOLOGRAM of the European Research Council (ERC), which is gratefully acknowledged.  The authors are most grateful to the anonymous referee for very helpful comments that have led to marked improvements.}

\subjclass[2010]{Primary 30D05, 37F10, 37F20}

\date{\today}

\maketitle
\tableofcontents

\section{Introduction}
\index{Newton map}\index{classification!of postcritically finite Newton maps}
The past four decades have seen tremendous progress in the understanding of holomorphic dynamics.  This is largely due to the fact that the complex structure provides enough rigidity, allowing many interesting questions to be reduced to tractable combinatorial problems.

To understand the dynamics of rational maps, an important first step is to understand the dynamics of postcritically finite maps\index{postcritically finite map}\index{postcritically finite}, namely the maps where each critical point has finite forward orbit.  Thurston's ``Fundamental Theorem of Complex Dynamics'' \cite{DH93} is available in this setting, providing an important characterization and rigidity theorem for postcritically finite branched covers that arise from rational maps.  Also, the postcritically finite maps are the structurally important ones, and conjecturally, the set of maps that are quasiconformally equivalent (in a neighborhood of the Julia set) to such maps are dense in parameter spaces \cite[Conjecture 1.1]{MC}.  

Polynomials form an important and well-understood class of rational functions.  In this case, the point at infinity is completely invariant, and is contained in a completely invariant Fatou component.  This permits enough dynamical structure so that postcritically finite polynomials may be described in finite terms, e.g.,\ using external angles at critical values or finite Hubbard trees.  A complete classification of postcritically finite polynomials has been given \cite{BFH,Poirier}\index{postcritically finite polynomial}\index{postcritically finite}\index{classification!of postcritically finite polynomials}.

Classification results for families of rational functions are rare and mostly concern one-dimensional families. There is a recent classification of critically fixed rational maps \cite{CGNPP,Hl19} and critically fixed anti-rational maps \cite{Gey,LLM} in arbitrary degree, but these classification results do not address higher period critical points.   The noteworthy family that exceeds all these limitations is the family classified here: Newton maps.  

\begin{definition}[Newton map]
A rational function $f:\widehat{\C}\to\widehat{\C}$ is called a \emph{Newton map} if there is some complex polynomial $p(z)$ so that $f(z)=z-\frac{p(z)}{p'(z)}$ for all $z\in\mathbb{C}$. \index{Newton map}
\end{definition}

Denote such a Newton map by $N_p$.  Newton maps of degree 1 and 2 are trivial and thus excluded from our entire discussion.  

Note that $N_p$ is precisely the function that is iterated when Newton's method is used to find the roots of the polynomial $p$.  Each root of $p$ is an attracting fixed point of the Newton map, and the point at infinity is a repelling fixed point.  The algebraic number of roots of $p$ is the degree of $p$, while the geometric number of roots of $p$ (ignoring multiplicities) equals the degree of $N_p$.  The space of degree $d$ Newton maps considered up to affine conjugacy has $d-2$ degrees of freedom, given by the location of the $d$ roots of $p$ after affine conjugation.  The space of degree $d$ complex polynomials up to affine conjugacy has $d-1$ degrees of freedom, given by the $d+1$ coefficients after affine conjugation. Thus it is clear that Newton maps form a substantial subclass of rational maps, making the combinatorial classification all the more remarkable. A brief summary of ``extended Newton graphs'' is given below in the introduction, where it should be noted that arbitrary choices must be made in the construction of so-called ``Newton rays'' with the result that there is no unique extended Newton graph associated with a Newton map.  The abstract graph definition and combinatorial equivalence are given in Definitions \ref{Def_AbstractExtNewtGraph} and \ref{Def_ThurstonEquivalenceAbstExtGraphs}).   

\begin{maintheorem}[Classification of postcritically finite Newton maps]
\label{Thm_Bijection}\index{classification!of postcritically finite Newton maps}\index{postcritically finite Newton map}\index{Newton map!postcritically finite}\index{Newton map!classification} There is a natural bijection between the set of postcritically finite Newton maps (up to affine conjugacy) and the set of abstract extended Newton graphs (up to combinatorial equivalence)  so that for every abstract extended Newton graph $(\Sigma, f)$, the associated postcritically finite Newton map has the property that any associated extended Newton graph is equivalent to $(\Sigma,f)$.
\end{maintheorem}

In the study of Newton maps, an important first theorem is the following characterization in terms of fixed point multipliers.

\begin{proposition}[Head's theorem]\cite{He}
\label{Prop:Head'sTheorem}\index{Head's theorem}
A rational map $f$ of degree $d\geq 3$ is a Newton map if and only if for each fixed point $\xi\in\C$, there is an integer $m\geq 1$ so that $f'(\xi)=(m-1)/m$.
\end{proposition}

This condition on multipliers forces $\infty$ to be a repelling fixed point by the holomorphic fixed point formula. In fact, for a postcritically finite Newton map all the finite fixed points must be superattracting (otherwise, the immediate basin contains a critical point that converges to the root), and this corresponds to the roots of $p$ being simple. 

There are a number of partial classification theorems for postcritically finite Newton maps.  Tan Lei has given a classification of cubic Newton maps in terms of matings and captures\index{classification!of cubic Newton maps} (or alternatively in terms of abstract graphs \cite{TL}; see also earlier work by Head \cite{He}).  Luo produced a similar combinatorial classification for Newton maps of arbitrary degree subject to the condition that there is only a single non-fixed critical value, and this critical value is either periodic or eventually maps to a fixed critical point \cite{Lu1}. 

The classification of postcritically \emph{fixed} Newton maps\index{postcritically fixed Newton map}\index{Newton map!postcritically fixed}\index{classification!of postcritically fixed Newton maps} for arbitrary degree (those Newton maps whose critical points eventually land on fixed points) is given in \cite{DMRS} building on the work of \cite{RueckertThesis}.  The fundamental piece of combinatorial data is the \emph{channel diagram} $\Delta$\index{channel diagram} which is constructed in \cite{HSS}.  This is a graph in the Riemann sphere whose vertices are given by the fixed points of the Newton map and whose edges are given by all accesses of the immediate basins of roots connecting the roots to $\infty$ (see the solid lines of Figure \ref{Fig_Deg6NewtonGraph}).  To capture the behavior of non-periodic critical points that eventually map to the channel diagram, it is natural to consider\footnote{We denote the $n$-th iterate of a dynamical system $f:X\to X$ by $f^n:X\to X$.} the graph $N_p^{-n}(\Delta)$ for some integer $n$.  However this graph is not necessarily connected (see Figure \ref{Fig_Deg6NewtonGraph} for an example), and so the \emph{Newton graph of level $n$}\index{Newton graph} associated with $N_p$ is defined to be the component of  $N_p^{-n}(\Delta)$ that contains $\Delta$. It is shown in \cite{DMRS}  that for any postcritically \emph{finite} Newton map $N_p$, there is some level $n$ so that the Newton graph of level $n$ contains all critical points that eventually map to the channel diagram (this fact is non-trivial because preimage components of the channel diagram were discarded).  For minimal $n$ this component is called the \emph{Newton graph} in the context of postcritically fixed maps, and the data consisting of this graph equipped with a graph map inherited from the dynamics of the Newton map is enough to classify postcritically \emph{fixed} Newton maps. 
\index{classification!of postcritically fixed Newton maps}

\begin{figure}[h]
\centerline{\includegraphics[width=135mm]{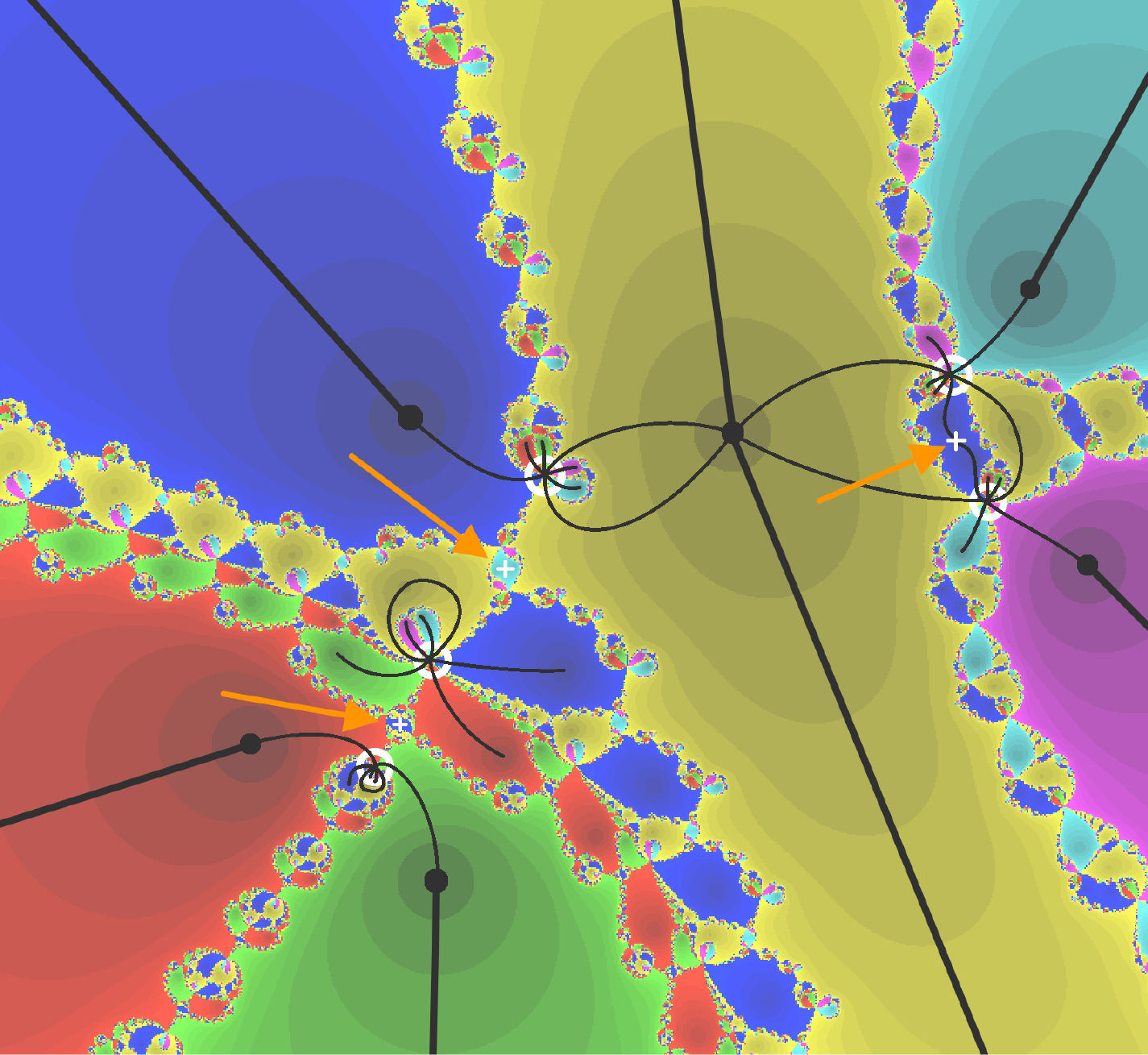}}
\caption{Dynamical plane of a degree 6 Newton map $N_p$. The six roots are indicated by black dots. The five finite poles are indicated by white circles. The channel diagram $\Delta$ is drawn with thick black curves, and $N_p^{-1}(\Delta)\setminus\Delta$ is drawn with thin black curves. The Newton graph\index{Newton graph} of level one $\Delta_1$ is visible as the component of $N_p^{-1}(\Delta)$ that contains $\Delta$. Note that $N_p^{-1}(\Delta)\setminus\Delta_1$ is nonempty and contains one connected component.  There are three non-fixed simple critical points indicated by orange arrows and a white ``+''. The rightmost such critical point is mapped by $N_p$ to the root in the blue basin, but more than one iterate is required to map the other two critical points to a root. There are no free critical points.}
\label{Fig_Deg6NewtonGraph}
\end{figure}

We classify postcritically \emph{finite} Newton maps,\index{postcritically finite Newton map}\index{Newton map!postcritically finite}\index{postcritically finite}\index{classification!of postcritically finite Newton maps} building on work of \cite{Mik}. The chief difficulty in this generalized setting is the existence of critical points whose forward orbit does not contain a fixed point. We thus call a critical point \emph{free}\index{free critical point} if it is not contained in the Newton graph $\Delta_n$ for any level $n$. In \cite{LMS1} a finite graph containing the postcritical set was constructed for a postcritically finite Newton map. The graph is composed of three types of pieces:
\begin{itemize}
\item  the Newton graph\index{Newton graph} (which contains the channel diagram\index{channel diagram}) is used to capture the behavior of critical points that are eventually fixed.  
\item Hubbard trees\index{Hubbard tree} are used to give combinatorial descriptions of renormalizations at periodic non-fixed postcritical points. See \cite{DLSS} for the construction of the renormalization.  Preimages of the Hubbard trees are taken to capture the behavior of critical points whose orbits intersect the Hubbard trees (or equivalently the free critical points).
\item Newton rays\index{Newton ray} (single edges comprised of either an internal ray or a sequence of infinitely many preimages of channel diagram edges) are used to connect all Hubbard trees and their preimages to the Newton graph.  
\end{itemize}

The construction of these three types of edges is given in \cite{LMS1} and not reproduced here, but an example is provided in Figures \ref{Pic_CubicNewtonRay} and \ref{CubicNewtonRayMagn}.

\begin{figure}[h!]
\centering
\includegraphics[width=12cm]{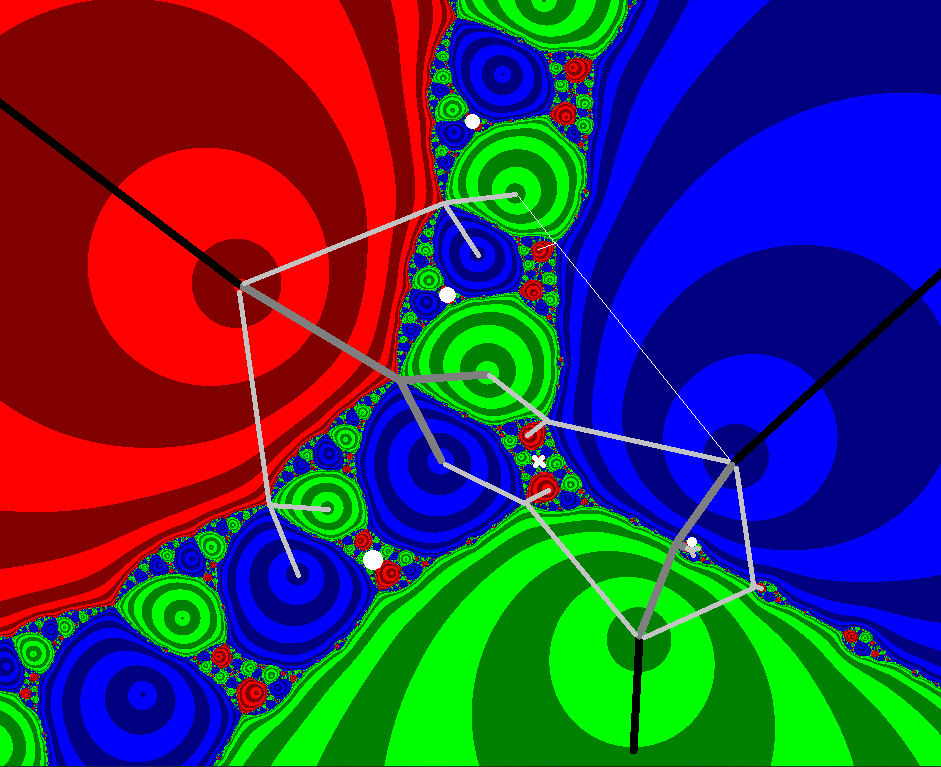}
\caption{\label{Pic_CubicNewtonRay} The dynamical plane of a cubic Newton map $N_p$ displaying part of the extended Newton graph.  The centers of the biggest red, green, and blue basins are fixed critical points. The white ``\textsf{x}'' indicates a free critical point, and its orbit is indicated by white dots. It has period 5 and the corresponding polynomial-like map straightens to $z\mapsto z^2$ (which is visible in Figure \ref{CubicNewtonRayMagn}).  Thick black edges indicate the channel diagram $\Delta$, and successively lighter edges indicate the additional edges in $\Delta_1$, $\Delta_2$ and $\Delta_3$ (due to the small scale many edges from $\Delta_3$ are omitted, though enough are drawn to separate the filled Julia sets as required by the construction).   The combinatorial invariant for $N_p$ consists of $\Delta_3$, the Hubbard trees of each of the five filled Julia sets, and five Newton rays connecting them to $\Delta_3$. A Hubbard tree/Newton ray pair is exhibited in the zoom of Figure \ref{CubicNewtonRayMagn}.  (Both images produced by Wolf Jung's Mandel software)}
\end{figure}

\begin{figure}[h!]
\centering
\includegraphics[width=12cm]{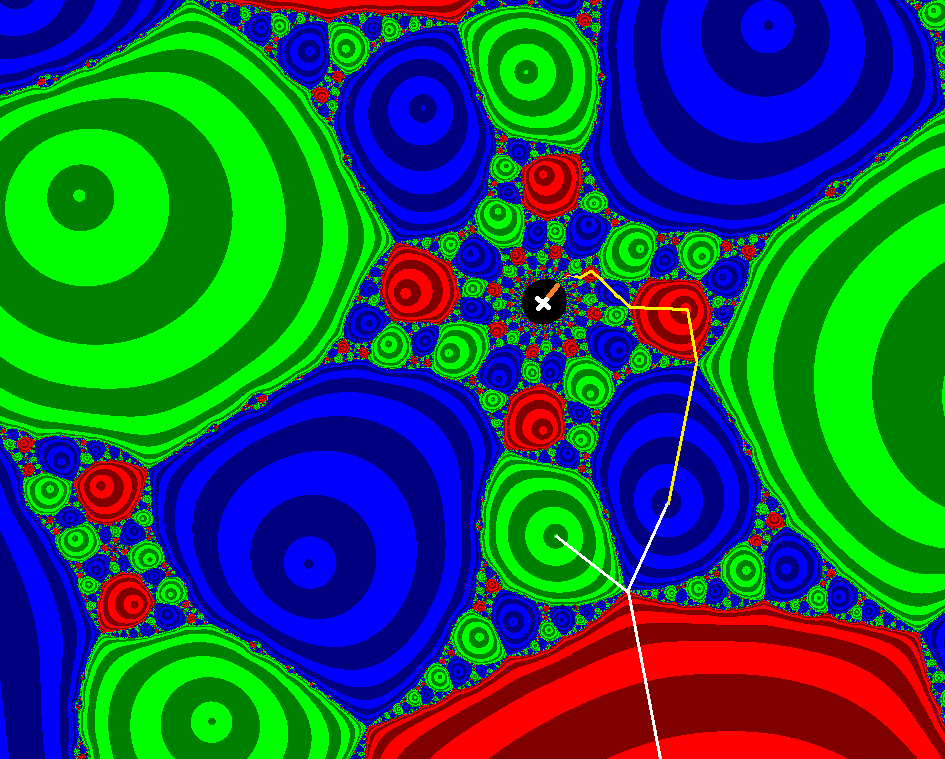}
\caption{\label{CubicNewtonRayMagn} Zoom of Figure \ref{Pic_CubicNewtonRay} at the free critical point. The filled Julia set is visible as a black disk, and its Hubbard tree is drawn in orange.  The white edges are level 3 edges in the Newton graph $\Delta_3$.  The yellow edge is a period 5 Newton ray.}
\end{figure}

The restriction of the Newton map to this ``extended Newton graph''\index{extended Newton graph}\index{Newton graph!extended} yields a graph self-map, and the resulting dynamical graphs are axiomatized (as abstract extended Newton graphs; see Definition \ref{Def_AbstractExtNewtGraph}).
\begin{theorem} [{Newton maps to graphs; \cite[Theorem 1.2]{LMS1}}] \label{Thm_NewtMapsGenerateExtNewtGraphs} 
For any extended Newton graph $\Delta^*_\N\subset\widehat{\mathbb{C}}$ associated with a postcritically finite Newton map $N_p$\index{postcritically finite Newton map}\index{Newton map!postcritically finite}\index{postcritically finite}, the pair $(\Delta^*_\N,N_p)$ satisfies the axioms of an abstract extended Newton graph.\index{abstract extended Newton graph}\index{Newton graph!abstract extended}

\end{theorem}

It must be emphasized that arbitrary choices were made in the construction of the Newton rays, necessitating a rather subtle but natural combinatorial equivalence relation on our way to a classification.  

Our first main result on the way to the classification is that every abstract extended Newton graph is realized by a unique Newton map (up to affine conjugacy); it is proven using Thurston's theorem.  In the following theorem statement, $\ol{f}$ denotes the unique extension (up to Thurston equivalence) of the graph map $f$ to a branched cover of the whole sphere, and the set of vertices of a graph $\Gamma$ is denoted by $\Gamma'$. 

\begin{maintheorem} [Graphs to Newton maps]
\label{Thm_Realization}  Let $(\Sigma,f)$ be an abstract
extended Newton graph (as in Definition \ref{Def_AbstractExtNewtGraph}). Then there is a postcritically finite
Newton map $N_p$, unique up to affine conjugacy, with extended Newton graph $\Delta^*_\N$ so that
the marked branched covers $(\ol{f},\Sigma')$ and $(N_p,(\Delta^*_\N)')$ are
Thurston equivalent.
\end{maintheorem}

Denote by $\Newt$ the set of postcritically finite Newton maps up to affine conjugacy, and by $\AENG$ we denote the set of abstract extended Newton graphs under the graph equivalence of Definition \ref{Def_ThurstonEquivalenceAbstExtGraphs}.  It follows from the statements of Theorem \ref{Thm_NewtMapsGenerateExtNewtGraphs} and \ref{Thm_Realization} that there are well-defined maps 
\[
\mathcal{F}: \Newt \to \AENG
\quad\text{ and }\quad
\mathcal{F}': \AENG \to \Newt
\]
respectively.  It will be shown that the mappings $\mathcal{F}$ and $\mathcal{F}'$ are bijective, and inverses of each other, yielding Main Theorem \ref{Thm_Bijection}.

\begin{remark}
This paper not only provides a classification of the largest non-polynomial family of rational maps so far, it also lays foundations for classification and rigidity results\index{rigidity (of Newton maps)} in a substantially larger context. In particular, there is the following rigidity result: all Newton maps, postcritically finite or not, are \emph{rigid} in the sense that any two such maps can be distinguished in purely combinatorial terms (plus conformal invariants such as multipliers of attracting cycles), except when they admit embedded polynomial-like dynamics that fails to be rigid \cite{DS} (see also \cite{RYZ} in the non-renormalizable case). In parallel, strong results about local connectivity for the Julia sets of Newton maps are developed. In particular, the boundary of every component of the basin of a root is locally connected; this was also shown independently in \cite{WYZ}.

The fundamental property of the dynamics that is underlying our work is that Fatou components have a common accessible boundary point at infinity, as well as the preimages of these Fatou components. A basic ingredient in more general classification and rigidity results builds on periodic Fatou components with common accessible boundary points, and for these our methods will be a key ingredient.

We also mention work of Mamayusupov \cite{Ma18} that establishes a bijection between the set of rational maps that arise as Newton maps of transcendental entire functions and the set of postcritically finite Newton maps of our study. Finally, we should also mention that Newton's method is much better at actually finding roots of complex polynomials than its reputation sometimes predicts; see for instance \cite{SSt,RSS,SCRSSS,S}.

\end{remark}

\medskip

\noindent\textbf{Structure of this paper}: Section \ref{Sec_ThurstonTheory} introduces Thurston's characterization and rigidity theorem for postcritically finite branched covers.  This theorem asserts that a topological branched cover that has no obstructing multicurves is uniquely realized by a rational map (under a mild assumption that is irrelevant for our purposes), and that such multicurves are the only possible obstructions for existence.  Since it is often very hard to show directly that a cover is unobstructed, we describe a theorem of Pilgrim and Tan that is very useful for this purpose,  controlling the location of obstructions.  Section \ref{Sec_ExtendingMapsOnGraphs} presents a result on how to extend certain kinds of graph maps to branched covers on the whole sphere.

Section \ref{Chap_AbstrExtNewtGraph} defines the abstract extended Newton graph, which will be shown to be a complete invariant for postcritically finite Newton maps.  The equivalence on such graphs is defined in Section \ref{Chap_GraphEquivalence}, and the connection between this combinatorial equivalence and Thurston equivalence is described.

Section \ref{NewtonMapsfromGraphs} proves Theorem \ref{Thm_Realization} by showing that abstract extended Newton graphs equipped with their graph self-maps extend to branched covers of the sphere that are unobstructed.  

Section \ref{Sec_ProofThmBijection} proves Theorem \ref{Thm_Bijection}, establishing the combinatorial classification of postcritically finite Newton maps.


\section{Thurston theory on branched covers} \label{Sec_ThurstonTheory} 
We will be using Thurston's theorem to prove that the combinatorial model for postcritically finite Newton maps is realized by a rational map, and we present the requisite background in this section.  As one observes from the statement of Thurston's theorem below, this amounts to showing that the combinatorial model has no obstructing multicurves.  There are infinitely many multicurves in a sphere with four or more marked points, so a priori it is very hard to show obstructions do not exist.  However, the ``arcs intersecting obstructions'' theorem of Pilgrim and Tan can in some cases drastically reduce the possible locations of obstructions.

Let $f:\S^2 \to \S^2$ be an orientation-preserving branched cover from the two-sphere to itself.  Denote the set of critical points of $f$ by $C_f$.  Define the postcritical set $P_f$\index{postcritical set} as follows:
\[P_f:=\bigcup_{n\geq 1}f^{n}(C_f).\]
The map $f$ is said to be \emph{postcritically finite}\index{postcritically finite} if the set $P_f$ is finite.

A \emph{marked branched cover} is a pair $(f,X)$, where $f:\S^2
\to \S^2$ is an orientation-preserving branched cover and $X$ is a
finite set containing $P_f$ such that $f(X) \subset X$.

\begin{definition}[Thurston equivalence of marked branched covers]
\label{Def_Thurston}
Two marked branched covers  $(f,X)$ and $(g,Y)$ are \emph{Thurston equivalent}\index{Thurston equivalent}\index{equivalent!Thurston} if there are two orientation-preserving homeomorphisms $\phi_1,\,\phi_2:\S^2\to\S^2$ such that
\[
    \phi_1\circ f = g\circ \phi_2
\]
and there exists an homotopy $\Phi: [0,1]\times\S^2\to\S^2$ with
$\Phi(0,\cdot)=\phi_1$ and $\Phi(1,\cdot)=\phi_2$ such that $\Phi(t,\cdot)|_{X}$
is constant in $t\in [0,1]$ with $\Phi(t,X)=Y$.  If $\phi_1$ and $\phi_2$ are homotopic to the identity map, the marked branched covers are said to be \emph{homotopic}.
\end{definition}

It is perhaps more intuitive to rephrase Thurston equivalence as saying that $f$ and $g$ are given only up to homotopy rel the marked points, and this homotopy may be chosen so that $f$ and $g$ are topologically conjugate.     

We say that a simple closed curve $\gamma$ is a \emph{simple closed curve in $(\S^2,X)$} if $\gamma \subset \S^2 \setminus X$.  Such a $\gamma$ is {\it
essential}\index{essential simple closed curve}\index{simple closed curve!essential} if both components of the complement $\S^2 \setminus
\gamma$ contain at least two points of $X$. 
Let $\gamma_0,\, \gamma_1$ be two simple closed curves in $(\S^2,X)$. We say that $\gamma_0$ and $\gamma_1$ are \emph{isotopic relative to $X$}, written $\gamma_0 \simeq_X \gamma_1$, if there exists a continuous, one-parameter family of simple closed curves in $(\S^2,X)$ joining $\gamma_0$ and $\gamma_1$. We use $[\gamma]$ to denote the isotopy class of a simple closed curve $\gamma$. A \emph{multicurve}\index{multicurve} is a collection of pairwise disjoint and non-isotopic essential simple closed curves in $(\S^2,X)$. A multicurve $\Pi$ is said to be $f$-stable\index{$f$-stable multicurve}\index{multicurve!$f$-stable}
if for every $\gamma \in \Pi$, every essential connected component
of $f^{-1}(\gamma)$ is isotopic relative to $X$ to some element of
$\Pi$.

\begin{definition}[Thurston linear transform]
For every $f$-stable multicurve $\Pi$ we define the
corresponding \emph{Thurston linear transform}\index{Thurston linear transform} $f_\Pi: \mathbb{R}^\Pi
\to \mathbb{R}^\Pi$ as follows:
\[
f_\Pi(\gamma)=\sum_{\gamma' \subset f^{-1}(\gamma)} \frac
{1}{\deg(f|_{\gamma'}: \gamma' \to \gamma)} [\gamma'],
\]
where $[\gamma']$ denotes the element of $\Pi$ isotopic to $\gamma'$ if it exists. If there are no such elements, the sum is taken to be zero.  Denote by $\lambda_\Pi$ the largest eigenvalue of $f_\Pi$  (by the Perron--Frobenius theorem, it exists and is non-negative real).
\end{definition}

The Thurston linear transform is also known equivalently as the \emph{Thurston matrix} or \emph{multicurve matrix}.\index{Thurston matrix}\index{multicurve matrix}

Suppose that $\Pi$ is a stable multicurve. A multicurve $\Pi$ is called a \emph{multi\-curve obstruction}\index{multicurve obstruction}\index{obstruction!multicurve} (or \emph{Thurston obstruction})\index{Thurston obstruction}\index{obstruction!Thurston} if
$\lambda_\Pi \geq 1$. A real-valued $n\times n$ matrix $A$ is called
\emph{irreducible} if for every entry $(i,j)$, there exists an integer
$k>0$ such that $A^k_{i,j}>0$. A multicurve $\Pi$ is said to be \emph{irreducible} if the matrix
representing the linear transform $f_\Pi$ is irreducible.

The statement of Thurston's theorem uses the notion of a hyperbolic orbifold.  We omit the definition, referring the reader to \cite{DH93} while observing that there are only a few well-understood cases where $O_f$ is not hyperbolic, and that $O_f$ is always hyperbolic if $f$ has at least three fixed branched points.  The latter is always the case for Newton maps of degree $d\geq 3$, so the restriction to hyperbolic orbifolds is of no concern to us.

\begin{theorem}[Thurston's theorem \cite{DH93, BCT}] \label{Thm_Thurston'sTheorem}\index{Thurston's theorem (on rational maps)}
A marked branched cover $(f,X)$ with
hyperbolic orbifold is Thurston equivalent to a marked rational map if and only if $(f,X)$ has no multicurve obstruction. Furthermore, if $(f,X)$ is unobstructed, the marked rational map is unique up to M\"obius conjugacy.
\end{theorem}

We now present a theorem of Pilgrim and Tan \cite{PT} that will be used in Section \ref{NewtonMapsfromGraphs} to show that certain marked branched covers arising from graph maps do not have obstructions and are therefore equivalent to rational maps by Thurston's theorem.  First some notation will be introduced.  

Assume that $(f,X)$ is a marked branched cover of degree $d\geq 3$.  An {\em arc} in $(\S^2,X)$ is a continuous map $\alpha:[0,1]\to \S^2$ such
that $\alpha(0)$ and $\alpha(1)$ are in $X$, the map $\alpha$ is injective on $(0,1)$, and $X\cap\alpha((0,1))=\emptyset$.  A set of pairwise non-isotopic arcs in $(\S^2,X)$ is called an
{\em arc system}.\index{arc system} 

The following intersection number is used in the statement of Theorem~\ref{Thm_ArcsInterOb}; we use the symbol $\simeq$ to denote isotopy relative to $X$.

\begin{definition} [Intersection number] \label{Def_IntersectionNumber}
Let $\alpha$ and $\beta$ each be an arc or a
simple closed curve in $(\S^2,X)$. Their {\em intersection
number} is
\[
    \alpha\cdot\beta := \min_{\alpha'\simeq\alpha,
    \,\beta'\simeq\beta} \#\{(\alpha'\cap\beta')\setminus X\}\;.
\]
This intersection can be extended to arc systems and
multicurves as follows: let $A$ and $B$ each be an arc system or a multicurve in $(\S^2,X)$. Then
\[
    A\cdot B := \min_{A'\simeq A,
    \,B'\simeq B} \#\{(A'\cap B')\setminus X\}\;.
\]
\end{definition}

For an arc system $\Lambda$, we introduce a linear map $f_{\Lambda}:\mathbb{R}^{\Lambda}\rightarrow\mathbb{R}^{\Lambda}$, which is a rough analogue of the Thurston linear map for multicurves.  For $\lambda\in\Lambda$, let
\[
    f_{\Lambda}(\lambda):=\sum_{\lambda'\subset f^{-1}(\lambda)} [\lambda']_{\Lambda}\;,
\]
where $[\lambda']_\Lambda$ denotes the element of $\Lambda$ homotopic to $\lambda'$ rel $X$ (the sum is taken to be zero if there are no such elements).
It is said that $\Lambda$ is {\em irreducible} if the matrix representing $f_{\Lambda}$ is irreducible.  

Denote by $\tilde{\Lambda}(f^{n})$ the union of those components of
$f^{-n}(\Lambda)$ that are isotopic to elements of $\Lambda$
relative $X$, and define $\tilde{\Pi}(f^{n})$ for a multicurve $\Pi$ analogously.  The following is a special case of a theorem from \cite{PT} that gives control on the location of irreducible multicurve obstructions by asserting that they may not intersect certain preimages of irreducible arc systems.

\begin{theorem} [Arcs intersecting obstructions {\cite[Theorem~3.2]{PT}}]\label{Thm_ArcsInterOb}\index{arcs intersecting obstructions}
Let $(f,X)$ be a marked branched cover,
$\Pi$ an irreducible multicurve obstruction, and $\Lambda$ an
irreducible arc system. Suppose furthermore that
$\#(\Pi\cap\Lambda)=\Pi\cdot\Lambda$. Then exactly one of
the following is true:
\begin{enumerate}
\item $\Pi\cdot\Lambda=0$ and $\Pi\cdot f^{-n}(\Lambda)=0$
for all $n\geq 1$.
\item $\Pi\cdot\Lambda\neq 0$ and
for $n\geq 1$, each component of $\Pi$ is isotopic to a unique
component of $\tilde{\Pi}(f^{n})$. The mapping $f^{
n}:\tilde{\Pi}(f^{ n})\to \Pi$ is a homeomorphism and
$\tilde{\Pi}(f^{n})\cap (f^{-n}(\Lambda)-
\tilde{\Lambda}(f^{ n}))=\emptyset$. More precisely, for each $\gamma\in\Gamma$, there is exactly one curve $\gamma'\subset f^{-n}(\gamma)$ such that $\gamma'\cap\tilde{\Lambda}(f^n)\neq\emptyset$. Moreover, the curve $\gamma'$ is the unique component of $f^{-n}(\gamma)$ which is isotopic to an element of $\Pi$.

\end{enumerate}
\end{theorem}


\section{Extending maps on finite graphs} \label{Sec_ExtendingMapsOnGraphs} We present a sufficient condition under which a certain type of map of a graph in $\S^2$ has a unique continuous extension to the whole sphere up to equivalence. The following formulation follows \cite[Chapter 5]{BFH}.  The Alexander Trick is foundational to such results and will be used elsewhere.

\begin{lemma}[Alexander trick]
\label{Lem:AlexanderTrick}\label{lem:AlexanderTrick}
\index{Alexander's trick}
 Let $h:\S^1 \to \S^1$ be an
orientation-preserving homeomorphism. Then there exists an orientation-preserving homeomorphism $\ol{h}:\overline{\D} \to \overline{\D}$ such that $\ol{h}|_{\S^1}
=h$. The map $\ol{h}$ is unique up to isotopy relative $\S^1$.
\end{lemma}

\begin{definition}[Finite graph] Let $V$ be a finite set of distinct points in $\S^2$. Each element of $V$ is called a \emph{vertex}. An \emph{edge} is a subset of $\S^2$ of the form $\lambda(I)$ where $I=[0,1]$ and
\begin{itemize}
 \item $\lambda:I\to\S^2$ is continuous and injective on $(0,1)$, and
\item $\lambda(x)\in V \iff x \in \{0,1\}$.
\end{itemize}
Let $E$ be a finite set of edges that (pairwise) intersect only at vertices.
A \emph{finite graph} (in $\S^2$) is a pair of the form $(V,E).$
\end{definition}
 We sometimes omit the reference to the ambient space $\mathbb{S}^2$ though it is always implicit.

\begin{definition}[Subgraphs]
Let $\Gamma_1=(V_1,E_1)$ and $\Gamma_2=(V_2,E_2)$ be finite graphs. We say that $\Gamma_1$ is a \emph{subgraph} of $\Gamma_2$ (denoted $\Gamma_1\subset\Gamma_2$) if $V_1\subset V_2$ and $E_1\subset E_2$. 
\end{definition}

\begin{definition}[Graph map]
Let $\Gamma_1,\Gamma_2$ be connected finite graphs. A continuous map $f:\Gamma_1 \to \Gamma_2$ is called a \emph{graph map}\index{graph map} if it is injective on each edge of $\Gamma_1$, if $f(V_1)\subset V_2$ and $f^{-1}(V_2)\subset V_1$, and $f$ is compatible with the embeddings of the graphs in $\S^2$.
\end{definition}

The compatibility condition on $f$ is a local condition at each vertex $v$, described as follows. If $f$ is locally injective at $v$, then $f$ is required to preserve the cyclic ordering at $v$. On the other hand, if $f$ is not locally injective at $v$, then the cyclic ordering of the half-edges at $v$ and $f(v)$ should be compatible with a local orientation-preserving cover of degree $\deg_v f$. We will use this definition only when the number of half-edges at $v$ equals $\deg_v f$ times the number of edges at $f(v)$, and thus compatibility means that the half-edges at $v$ have the same cyclic ordering as $f(v)$, repeated $\deg_v f$ times.

\begin{definition}[Regular extension]
\label{Def_GraphMap}\index{regular extension (of graph map)} Let $f:\Gamma_1\to\Gamma_2$ be a graph map. An
orientation-preserving branched cover $\ol{f}:\S^2\to\S^2$ is
called a \emph{regular extension} of $f$ if $\ol{f}|_{\Gamma_1}=f$
and $\ol{f}$ is injective on each component of $\S^2\setminus
\Gamma_1$.
\end{definition}
It follows that every regular extension $\ol{f}$ may have critical points only at the vertices of $\Gamma_1$, and the local degree of $\ol{f}$ at $v$ coincides with $\deg_v(f)$.
\begin{lemma}[Isotopic graph maps {\cite[Corollary 6.3]{BFH}}]
\label{Lem_IsotopGraphMaps}  Let
$f,g:\Gamma_1\to\Gamma_2$ be two graph maps that coincide on the
vertices of $\Gamma_1$ such that for each edge $e$ in $\Gamma_1$ we have $f(e)=g(e)$ as a set. Suppose that $f$ and $g$
have regular extensions $\ol{f},\ol{g}:\S^2\to\S^2$. Then there
exists a homeomorphism $\psi:\S^2\to\S^2$, isotopic to the identity
relative the vertices of $\Gamma_1$, such that $\ol{f}=\ol{g} \circ
\psi$.
\end{lemma}

We must establish some notation for the following proposition from \cite{BFH}.  Let $f:\Gamma_1\to\Gamma_2$ be a graph map. For each
vertex $v$ of $\Gamma_i$ with fixed $i\in\{1,2\}$, choose a neighborhood $U_v\subset\S^2$ such
that all edges of $\Gamma_i$ that enter $U_v$ are incident to $v$, the vertex $v$ is the only vertex of $\Gamma_i$ in $U_v$, and the neighborhoods $U_v$ and $U_w$ are disjoint for all vertices $v\neq w$ in $\Gamma_i$. We
may assume without loss of generality that in local coordinates, $U_v$ is a round disk of radius $1$ that is centered at $v$ and that the intersection of any edge of $\Gamma_i$ with $U_v$ is either empty or a radial line segment.  Without loss of generality, we may assume that $f|_{U_v\cap\Gamma_1}$ is length-preserving for all vertices $v$ in $\Gamma_1$.

We describe how to explicitly extend $f$ to each $U_v$.  For a vertex $v\in\Gamma_1$, let $\gamma_1$ and
$\gamma_2$ be two adjacent edges ending there. In local coordinates,
these are radial lines at angles $\Theta_1,\Theta_2$ where
$0<\Theta_2-\Theta_1\leq 2\pi$ (if $v$ is an endpoint of $\Gamma_1$,
then set $\Theta_1=0$, $\Theta_2=2\pi$). In the same way, choose
arguments $\Theta_1',\,\Theta_2'$ for the image edges in $U_{f(v)}$
and extend $f$ to a map $\tilde{f}$ on $\Gamma_1\cup\bigcup_v U_v$
defined by
\begin{equation}\label{eqn:localExtension}
\tilde{f}(\rho,\Theta)=\left(\rho,
\frac{\Theta_2'-\Theta_1'}{\Theta_2-\Theta_1}\cdot(\Theta-\Theta_1)+\Theta_1'\right),
\end{equation}
where $(\rho,\Theta)$ are polar coordinates in the sector bounded by
the rays at angles $\Theta_1$ and $\Theta_2$. In particular, sectors are
mapped onto sectors in an orientation-preserving way.

\begin{proposition}[{\cite[Proposition 5.4]{BFH}}]
\label{Prop_RegExt}  A graph map
$f:\Gamma_1\to\Gamma_2$ has a regular extension if and only if for
every vertex $y\in\Gamma_2$ and every component $U$ of
$\S^2\setminus\Gamma_1$, the extension $\tilde{f}$ is injective on
\[
    \bigcup_{v\in f^{-1}(y)} U_v \cap U\;.
\]
\end{proposition}

The fundamental combinatorial object in our classification of Newton maps is a finite graph $\Sigma$ equipped with a self-map $f:\Sigma\to\Sigma$ (Definition \ref{Def_AbstractExtNewtGraph}).  Strictly speaking,  $f$ is in general not a graph map since Newton ray edges contain finitely many preimages of vertices in the Newton graph that are not vertices in $\Sigma$ (these vertices are purposely ignored on our way to producing a finite graph).  This motivates the following weaker definition which is identical except that we no longer assume $f^{-1}(V_2)\subset V_1$.

\begin{definition}[Weak graph map]
 A continuous map $f:\Gamma_1 \to \Gamma_2$ is called a \emph{weak graph map}\index{weak graph map}\index{graph map!weak} if it is injective on each edge of $\Gamma_1$, if $f(V_1)\subset V_2$, and $f$ is compatible with the embeddings of the graphs in $\S^2$.
\end{definition}

\begin{remark}\label{rem_weakGraphMap}
Given a weak graph map $f:\Gamma_1 \to \Gamma_2$, the combinatorics of the domain can be slightly altered to produce a graph map $\hat{f}:\hat{\Gamma}_1\rightarrow\Gamma_2$ in the following natural way.  We take the graph $\hat{\Gamma}_1$ to have vertices given by $f^{-1}(V_2)$, and edges given by the closures of complementary components of $\Gamma_1\setminus f^{-1}(V_2)$. We simply take $\hat{f}=f$.
\end{remark}


\section{Abstract extended Newton graph} \label{Chap_AbstrExtNewtGraph} 
In \cite{LMS1}, we extracted from every postcritically finite Newton map an extended Newton graph (Section 6.1), and we axiomatized these graphs in Section 7.  In this section we review the definition of the abstract extended Newton graph which will be used in Section \ref{Sec_ProofThmBijection} of the present work to classify postcritically finite Newton maps. Abstract extended Newton graphs consist of three pieces: abstract Newton graphs, abstract extended Hubbard trees, and abstract Newton rays connecting the first two objects.  

The definition of abstract extended Hubbard trees was given in Definition~4.4 of \cite{LMS1}, and will not be repeated here.  We simply note that it is the usual definition of degree $d$ abstract Hubbard tree from \cite{Poirier}, where the set of marked points includes all periodic points of periods up to some fixed length $n$ (since postcritically finite Newton maps cannot have parabolic cycles, the number of periodic points of period $i$ equals $d^i$).  Such an abstract extended Hubbard tree is said to have \emph{cycle type $n$}.

To define the abstract Newton graph, it is necessary to first define the abstract channel diagram.

\begin{definition}
\label{Def_ChannelDiagram}
\index{abstract channel diagram}\index{channel diagram!abstract} An \emph{abstract channel diagram} of
degree $d\geq 3$ is a graph $\Delta \subset \S^2$ with vertices
$v_\infty, v_1,\dots,v_d$ and edges $e_1,\dots,e_l$ that satisfies the
following:
\begin{itemize}
\item $l\leq 2d-2$;
\item each edge joins $v_\infty$ to some $v_i$ for $i\in\{1,2,...,d\}$; 
\item each $v_i$ is connected to $v_\infty$ by at least one edge;
\item \label{necessaryCondition} if $e_i$ and $e_j$ both join $v_\infty$ to $v_k$, then each connected component of
$\S^2\setminus \ol{e_i\cup e_j}$ contains at least one vertex of
$\Delta$.
\end{itemize}
\end{definition}

The classification of postcritically fixed Newton maps was given in terms of a combinatorial object called the ``abstract Newton graph" \cite{DMRS}. We define the term almost identically except that in the following definition Condition (3) is relaxed from equality to an inequality (this corresponds to the fact that postcritically finite maps may have critical points that are not eventually fixed).

\begin{definition}[Abstract Newton graph]
\label{Def_NewtonGraph}
\index{abstract Newton graph}\index{Newton graph!abstract}
 Let $\Gamma$ be a connected finite
graph in $\S^2$ with vertex set $V(\Gamma)$ and $f:\Gamma\to\Gamma$ a
graph map. The pair $(\Gamma,f)$ is called an \emph{abstract Newton
graph of level $\mathcal{N}_{\Gamma}$} if it satisfies the following conditions:
\begin{enumerate}
\item[(1)] There exists $d_{\Gamma}\geq 3$ and an abstract channel diagram
$\Delta\subsetneq\Gamma$ of degree $d_\Gamma$ such that
 $f$ fixes each vertex and each edge of $\Delta$ (pointwise).

\item[(2)] If $v_\infty, v_1, \dots,v_{d_\Gamma}$ are the vertices of $\Delta$, then
$v_i\in \ol{\Gamma\setminus\Delta}$ if and only if $i\neq \infty$.
 Moreover, there are exactly $\deg_{v_i}(f)-1\geq 1$ edges in $\Delta$ that connect $v_i$ to $v_\infty$ for
 $i\neq \infty$.

\item[(3)] $\sum_{x\in V(\Gamma)} \left(\deg_x f-1\right) \leq 2d_{\Gamma}-2$. 

\item[(4)] $\N_\Gamma$ is the minimal integer so that $f^{\N_\Gamma-1}(v)\in\Delta$ for all $v\in V(\Gamma)$ with $\deg_v f>1$.

\item[(5)] $f^{\N_\Gamma}(\Gamma)\subset\Delta$.

\item[(6)] For every $v\in V(\Gamma)$ with $f^{\N_\Gamma-1}(v)\in\Delta$, the number of adjacent edges in $\Gamma$ equals $\deg_v f$ times the number of edges adjacent to $f(v)$.

\item[(7)] The graph $\overline{\Gamma\setminus\Delta}$ is connected.

\item[(8)] \label{Cond_Extension} For every vertex $y\in V(\Gamma)$ and every component $U$ of $\S^2\setminus\Gamma$, the local extension $\tilde{f}$ from Equation (\ref{eqn:localExtension}) is injective on
$\bigcup_{v\in f^{-1}(y)} U_v \cap U\;.$

\end{enumerate}
\end{definition}

Next we define abstract Newton rays.  Let $\Gamma$ be a finite connected graph embedded in $\S^2$ and $f: \Gamma \to \Gamma$ a weak graph map so that after $f$ is promoted to a graph map in the sense of Remark \ref{rem_weakGraphMap}, it can be extended to a branched cover $\overline{f}:\S^2 \to \S^2$. 

\begin{definition}[Abstract Newton ray]
\label{Def_AbstractNewtonRay}\index{abstract Newton ray}\index{Newton ray!abstract}
{Let $\mathcal{R}$ be an arc in $\S^2$ whose endpoints are denoted $i(\mathcal{R})$ and $t(\mathcal{R})$. Then $\mathcal{R}$ is called  an \emph{abstract Newton ray with respect to $(\Gamma,f)$} if $\mathcal{R} \cap \Gamma=\{i(\mathcal{R})\}$ and $\overline{f}(\mathcal{R})= \mathcal{R} \cup \mathcal{E}$, where $\mathcal{E}$ is a (possibly empty) subgraph of $\Gamma$. Such an abstract Newton ray is called a 
\emph{periodic abstract Newton ray with respect to $(\Gamma,f)$}} if moreover
there is a minimal positive integer $m$ so that $\overline{f}^m(\mathcal{R})= \mathcal{R} \cup \mathcal{E}$, where $\mathcal{E}$ is a (possibly empty) subgraph of $\Gamma$.
We say that the integer $m$ is \emph{the period} of $\mathcal{R}$, and that $\mathcal{R}$ \emph{lands} at $t(\mathcal{R})$. 
\end{definition}
\begin{definition}[Preperiodic abstract Newton ray]
\label{Def_AbstractNewtonRayPrePer} An abstract Newton ray $\mathcal{R'}$  is called 
a \emph{preperiodic abstract Newton ray with respect to $(\Gamma,f)$} if the following hold:
\begin{itemize}
\item there is a minimal integer $l>0$ such that $\overline{f}^l(\mathcal{R'}) = \mathcal{R} \cup \mathcal{E}$, where $\mathcal{E}$ is a (possibly empty) subgraph of $\Gamma$ and $\mathcal{R}$ is a periodic abstract Newton ray with respect to $(\Gamma,f)$.
\item $\mathcal{R'}$ is not a periodic abstract Newton ray with respect to $(\Gamma,f)$.
\end{itemize}
We say that the integer $l$ is \emph{the preperiod} of $\mathcal{R}'$, and that $\mathcal{R}'$ \emph{lands} at $t(\mathcal{R}')$.\end{definition}

Now we are ready to introduce the concept of an abstract extended Newton graph. Later we prove that this graph carries enough information to characterize postcritically finite Newton maps.

\begin{definition}[Abstract extended Newton graph]
\label{Def_AbstractExtNewtGraph}
\index{abstract extended Newton graph}\index{extended Newton graph!abstract}\index{Newton graph!abstract extended} 
Let $\Sigma \subset \S^2$ be a finite connected graph, and let $f:\Sigma \to \Sigma$ be a weak graph map.  A pair $(\Sigma,f)$ is called an \emph{abstract extended Newton graph} if
the following are satisfied:

\begin{enumerate}
\renewcommand{\theenumi}{(\arabic{enumi})}

\item \label{Cond_NewtonGrah} (Abstract Newton graph) There exists a positive integer $\N$ and an abstract Newton graph
$\Gamma$ at level $\N$ so that $\Gamma \subseteq \Sigma$. Furthermore $\N$ is minimal so that condition \ref{Cond_TreesSeparated} holds.

\item \label{Cond_PerHubbardTrees} (Periodic Hubbard trees) There is a finite collection of  (possibly degenerate) minimal abstract extended Hubbard trees $H_i \subset \Sigma$ which are disjoint from $\Gamma$, and for each $H_i$ there is a minimal positive
integer $m_i\geq 2$ called the \emph{period of the tree} such that $f^{m_i}\left(H_i\right)=H_i.$

\item \label{Cond_PrePerHubbardTrees} (Preperiodic trees) There is a finite collection of possibly degenerate trees 
$H'_i \subset \Sigma$ of preperiod $\ell_i$, i.e. there is a minimal positive integer $\ell_i$ so that $f^{\ell_i}(H'_i)$ is a periodic Hubbard tree ($H'_i$ is not necessarily a Hubbard tree).  Furthermore for each $i$, the tree $H'_i$ contains a critical or postcritical point.

\item \label{Cond_TreesSeparated} (Trees separated) Any two different periodic or pre-periodic Hubbard trees lie in different complementary components of $\Gamma$.

\item \label{Cond_NewtonRaysPer}
(Periodic Newton rays) For every periodic abstract extended Hubbard tree $H_i$ of period $m_i$, the graph $\Sigma$ contains exactly one periodic abstract Newton ray $\mathcal{R}_i$ with respect to $(\Gamma,f)$.  The ray lands at a repelling fixed point $\omega_i\in H_i$ of $f^{m_i}$ and has period $m_i$.

\item \label{Cond_NewtonRaysPreper}
(Preperiodic Newton rays) For every preperiodic tree $H'_i$, there exists at least one preperiodic abstract Newton ray in $\Sigma$ with respect to $(\Gamma,f)$ connecting a vertex of $H'_i$ to $\Gamma$.

\item \label{Cond_uniqueExtend} (Unique extendability) For every vertex $y\in V(\Sigma)$ and every component $U$ of $\S^2\setminus\Sigma$, the local extension $\tilde{f}$ from Equation (\ref{eqn:localExtension}) is injective on
$\bigcup_{v\in f^{-1}(y)} U_v \cap U$.

\item \label{Cond_DegreeExtGraph} (Topological admissibility) $\sum_{x\in V(\Sigma)} \left(\deg_x f-1\right) = 2d_{\Gamma}-2$, where $d_{\Gamma}$ is the degree of the abstract channel diagram $\Delta \subset \Gamma$.

\item \label{Cond_TypesOfEdges} (Edges and vertices) Every edge in $\Sigma$ must be one of the following three types:
\begin{itemize}
\item Type N: An edge in the abstract Newton graph $\Gamma$ of condition \ref{Cond_NewtonGrah}.
\item Type H: An edge in a periodic or pre-periodic abstract Hubbard tree of condition \ref{Cond_PerHubbardTrees} or \ref{Cond_PrePerHubbardTrees}.
\item Type R: A periodic or pre-periodic abstract Newton ray with respect to $(\Gamma,f)$ from condition \ref{Cond_NewtonRaysPer} or \ref{Cond_NewtonRaysPreper}.
\end{itemize}
As a consequence, every vertex of $\Sigma$ is either a Hubbard tree vertex or a Newton graph vertex.

\end{enumerate}

\end{definition}

\begin{remark}[Regular extension]
\index{regular extension}
 The purpose of condition \ref{Cond_uniqueExtend} is that after $f$ has been upgraded to a graph map following Remark \ref{rem_weakGraphMap}, the hypothesis of Proposition \ref{Prop_RegExt} is met. Thus $f$ has a regular extension $\ol{f}$ which is unique up to Thurston equivalence.
\end{remark}

\begin{remark}[Implied auxiliary edges]\label{rem:redundantArcs}  
Suppose that $H_i$ is a Hubbard tree (or Hubbard tree preimage) in some complementary component $U_i$ of $\Gamma$ with connecting Newton ray $R_i$.  If $H_i$ contains a critical point, the existence of a regular graph map extension from Condition \ref{Cond_uniqueExtend} implies that $\Sigma$ must have at least one pre-periodic Newton ray edge distinct from $R_i$ connecting $H_i$ to $\Gamma$.  All such pre-periodic rays must map to $f(R_i)$ under $f$ (ignoring the parts in $\Gamma$ as usual), and each such ray is called an \emph{auxiliary edge corresponding to $R_i$}.
\index{auxiliary edge}
\end{remark}

\begin{remark}(Consistency in \cite{LMS1}) It is well-known that a polynomial may have a fixed point that is the landing point of a non-fixed periodic cycle of rays. On the other hand, each polynomial has at least one fixed point that is the landing point of a fixed ray (in the quadratic case, this is the so-called $\beta$ fixed point). There was some effort in \cite{LMS1} to allow more generally that a periodic Newton ray land at a fixed point in a Hubbard tree through a higher period access. Unfortunately this was not done consistently, and so there are some matters that we would like to clarify. First, Condition (5) of Definition 7.3 in \cite{LMS1} might seem to permit accesses of higher period $r_i$, but the uniqueness of Newton rays in a given complementary component of $\Gamma$ (found in the same condition) immediately implies that $r_i=1$. Thus, despite the superficial difference, Condition (5) in Definition 7.3 of \cite{LMS1} and  Definition \ref{Def_AbstractExtNewtGraph} above are actually equivalent. The only remaining issue is that in the proof of Theorem 6.2 in \cite{LMS1}, one should always choose the fixed point of the first return map on the Hubbard tree to be the landing point of a fixed external ray under straightening. This ensures that the graph satisfies Condition \ref{Cond_NewtonRaysPer}.

\end{remark}


\section{Equivalence of abstract extended Newton graphs}\label{Chap_GraphEquivalence}

When the extended Newton graph was constructed for a postcritically finite Newton map in \cite{LMS1}, the Newton graph and Hubbard tree edges were constructed intrinsically, but the construction of the Newton rays involved many choices.  The endpoints and accesses of the Newton rays were chosen arbitrarily, and in the case of non-degenerate Hubbard trees, there are a countably infinite number of homotopy classes of arcs by which the tree could be connected to the Newton graph (corresponding to the fact that removing the Hubbard tree from the complementary component of the Newton graph produces a topological annulus).

 Let $(\Sigma_1,f_1)$ and $(\Sigma_2,f_2)$ be two abstract extended Newton graphs.  In this section, we define an equivalence relation for abstract extended Newton graphs so that we can tell from the combinatorics of $(\Sigma_1,f_1)$ and $(\Sigma_2,f_2)$ whether or not their extensions to branched covers are Thurston equivalent.    In fact, the main difficulty in determining equivalence of these graphs comes from establishing the equivalence of extensions across the topological annuli just mentioned. This motivates the following notation: for an abstract Newton graph $\Sigma$, denote by $\Sigma^-$ the resulting graph when all edges of type R are removed; only the type N and H edges remain.  We keep the endpoints of the removed edges as vertices of $\Sigma^-$.

The combinatorial equivalence given below in Definition \ref{Def_ThurstonEquivalenceAbstExtGraphs} must somehow encode the Thurston class of graph map extensions to complementary components of $\Sigma_1^-$ and $\Sigma_2^-$ that contain non-degenerate Hubbard trees (for other types of components it is already clear how to proceed because they are either topological disks or once-punctured disks).  This is our primary focus from now until the definition is given. 

\begin{definition}[Newton ray grand orbit]\index{Newton ray grand orbit}  The (forward) \emph{orbit} of a Newton ray $R$ in an abstract extended Newton graph $\Sigma$ is the set of all Newton rays $R'$ in $\Sigma$ so that $f^k(R)$ contains $R'$ for some $k$.  The \emph{grand orbit} of a Newton ray $R$ in $\Sigma$ is the union of all Newton rays $R'$ in $\Sigma$ whose orbit contains an edge in the orbit of $R$.
\end{definition}

\begin{remark}[Some simplifying assumptions]\label{rem:raySimplifying}
To simplify notation, we assume in Sections \ref{subsec_NewtonRayEndpoints} and  \ref{subsec_EquivNewtonRayGO} that  $\Sigma_1$ and $\Sigma_2$ are combinatorially and dynamically equal apart from their Newton rays.  Specifically this means that the identity map on $\S^2$ induces a graph isomorphism between the Newton graphs of $\Sigma_1$ and $\Sigma_2$ (from now on denoted $\Gamma$), as well as the Hubbard trees.  We also assume that $f_1=f_2$ on all vertices of $\Sigma_1$ and $\Sigma_2$ (the restriction to vertices of either graph map will be denoted $f$). 
\end{remark}

In Section \ref{subsec_NewtonRayEndpoints} we describe how to alter the endpoints and accesses of Newton ray grand orbits in $\Sigma_1$ so that they coincide with those of $\Sigma_2$ without changing the homotopy class of the graph map extension. Once this is done, a method is given in Section \ref{subsec_EquivNewtonRayGO} to determine whether the rays yield equivalent extensions across complementary components of $\Sigma_1^-$ and $\Sigma_2^-$ that contain non-degenerate Hubbard trees; accordingly, an equivalence is placed on ray grand orbits.  Finally the combinatorial equivalence of abstract extended Newton graphs is given in Section \ref{subsec_equivOnAENG} in terms of the equivalence on ray grand orbits.

\subsection{Making ray endpoints and accesses coincide}\label{subsec_NewtonRayEndpoints}

We present an initial alteration to the Newton ray grand orbits of $\Sigma_1$ and $\Sigma_2$ so that their endpoints and accesses to the Newton graph and Hubbard tree vertices coincide (it is possible and not infrequent for a repelling fixed point of a polynomial to be the landing point of multiple external rays, and we simply wish to fix a preferred external ray, corresponding to an access to the fixed point in the complement of the filled Julia set). These alterations are done so as to not change the homotopy classes of the graph map extensions, so after the alterations the different graph maps are easier to compare because they only differ in the homotopy classes of rays connecting Hubbard trees to the Newton graph with corresponding endpoints, and so can be distinguished using an integer condition. Even though the operation is performed on ray grand orbits, it is somewhat non-dynamical in nature because a periodic Newton ray may be replaced with a Newton ray that is not periodic.

\begin{lemma}\label{Lem:CompatibleRays}
Suppose that $\Sigma_1^-=\Sigma_2^-$ and that $f_1|_{\Sigma_1^-}=f_2|_{\Sigma_2^-}$. Let $H_{m}$ be a  Hubbard tree of period $m\geq 2$ in both $\Sigma_1$ and $\Sigma_2$, where $H_m$ is contained in the complementary component $U_m$ of $\Gamma$. Let $R_{1,m}\subset\Sigma_1$ be a Newton ray landing at $\omega_1\in H_{m}$ and $R_{2,m}\subset\Sigma_2$ a Newton ray landing at $\omega_2\in H_{m}$. Then there is a Newton ray $R_{2,m}'$ landing at $H_m$ whose grand orbit under the extension $\ol{f_2}$ has the same endpoints and accesses as the ray grand orbit of $R_{1,m}$ (while the homotopy class of $R'_{2,m}$ will usually be different from that of $R_{q,m}$).

\end{lemma}

\begin{proof}
Let $H_{m+1}$ be the Hubbard tree that contains $f(H_m)$, and let $V_{m}$ be the complementary component of $f(\Gamma)$ containing $H_{m+1}$. Note that upon restricting the domain,
\[\ol{f}_2:U_m\setminus \ol{f}_2^{-1}(H_{m+1})\to V_m\setminus H_{m+1}\]
 defines a covering map between two topological annuli. Under this map, the image of $R_{1,m}$ is a curve in $V_m$ that is not necessarily simple, but using annular coordinates this image curve is homotopic in $V_m\setminus H_{m+1}$ rel endpoints to a simple curve $\rho'_{2,m+1}$. Let $R'_{2,m}$ be the lift of $\rho'_{2,m+1}$ under $\ol{f}_2$ that is homotopic to $R_{1,m}$, having the same accesses and endpoints (such a lift exists by the homotopy lifting property rel the vertex set). Clearly $\ol{f}_2(R'_{2,m})$ is a simple arc in $V_m\setminus H_{m+1}$.  We have thus shown that $R'_{2,m}$ is a ray that has the correct endpoints and accesses. 
 
 Next we produce the ray grand orbit of $R'_{2,m}$ by a lifting procedure. Let $H_{m-1}$ be a Hubbard tree so that $f_2(H_{m-1})$ contains $H_m$, and let $R_{1,m-1}$ be the ray in the grand orbit of $R_{1,m}$ that lands at $H_{m-1}$. Define $U_{m-1}$ to be the complementary component of $\Gamma$ containing $H_{m-1}$, and let $V_{m-1}:=\ol{f}_2(U_{m-1})$. Let $z_m$ be the endpoint of $f_2(R_{1,m-1})$ that is in $\partial V_{m-1}$. Note that the portion of the path $f_2(R_{1,m-1})$ that is contained in $\Gamma$ defines a path connecting $z_m$ to the endpoint of $R'_{2,m}$ that lies in the Newton graph. The concatenation of this path with $R'_{2,m}$ yields a path in $\overline{V_{m-1}}$ that connects $H_m$ to $z_m$. There is a lift of this concatenation that has the same endpoints and accesses as $R_{1,m-1}$, and we denote this lift by $R'_{2,m-1}$.
 
Continuing this lifting procedure inductively produces the desired ray grand orbit.
\end{proof}

\begin{corollary}\label{cor:CompatibleRayOrbits}
Suppose that $\Sigma_1^-=\Sigma_2^-$ and that $f_1|_{\Sigma_1^-}=f_2|_{\Sigma_2^-}$. There is a graph $\Sigma_2'$ that is the domain of a continuous map $f_2'$ so that
\begin{itemize}
\item $\Sigma_2^-$ is a subgraph of $\Sigma_2'$ and $f_2|_{\Sigma_2^-}=f_2'|_{\Sigma_2^-}$,
\item the only edges of $\Sigma_2'$ that are not in $\Sigma_2^-$ are Newton ray edges, 
\item the endpoints and accesses of the Newton rays in $\Sigma_1$ coincide with those of the rays in $\Sigma_2'$, and
\item $f_2'$ is the restriction of $\ol{f}_2$ to $\Sigma_2'$. 
 \end{itemize}
\end{corollary}

\begin{proof}
Apply the preceding lemma to each Newton graph ray orbit in $\Sigma_1$.
\end{proof}

\begin{remark}
Having changed the Newton rays, a corresponding change is made to the auxiliary edges. The auxiliary edges associated with $R_{2,i}'$ are taken to be all the Newton rays in $\ol{f}_2^{-1}(\ol{f}_2(R_{2,i}))$ that are contained in $U_i$. It is easily seen that $\Sigma_2'$ in the Corollary satisfies all of the properties of an abstract extended Newton graph except that the Newton rays are not necessarily periodic or preperiodic.  The loss of periodicity is actually not significant since we only need to understand certain topological properties of the extension described next.
\end{remark}

\subsection{Equivalence on Newton ray grand orbits}\label{subsec_EquivNewtonRayGO}

We now wish to compare the extensions of graph maps over complementary components of $\Sigma_1^-$ and $\Sigma_2^-$ containing nondegenerate Hubbard trees in their closures.  The other complementary components may only be disks or once-punctured disks; these are not discussed here because there is a unique extension over such components up to isotopy given by the Alexander trick \cite[Chapter 2.2]{Fa}.  

Restricting attention to the complementary components of $\Gamma$ that contain the grand orbit of a single Hubbard tree, we will show in Lemma \ref{Lem:NumericsOfAnnulusExtensions} that whether or not two extensions are equivalent (in the sense of Definition \ref{defn:ThurstonEquivGraphExtensionAnnulus}) can be determined solely in terms of numerical properties of the Newton rays.  We can then define a combinatorial equivalence on Newton ray grand orbits so that two ray grand orbits are Thurston equivalent if and only if the extensions to the complementary components of $\Gamma$ intersecting the grand orbit are equivalent (see Lemma \ref{Lem_GlobalExtOfGraphEquiv}).

Let $H_1$ be a nondegenerate Hubbard tree in $\Sigma_1$ (and $\Sigma_2$) of preperiod $r\geq 0$ and period $m\geq2$, and let $H_i=f^{i-1}(H_1)$ for $1\leq i\leq r+m$.  Let $U_i$ be the complementary component of $\Gamma$ that contains $H_i$ and let $\mathscr{U}=\cup_i U_i$. Each $\ol{f_1}(U_i)$ is homeomorphic to a disk, and is a complementary component of $f(\Gamma)$ that contains both $H_{i+1}$ and $U_{i+1}$ as a proper subset. So strictly speaking, the restriction of $\ol{f_1}$ and $\ol{f_2}$ to $\mathscr{U}$ do not define a dynamical system on $\mathscr{U}$. To remedy this, fix a homeomorphism $\ol{f_1}(\mathscr{U})\to \mathscr{U}$ that restricts to the identity on each Hubbard tree. Postcomposing each of $\ol{f_1}$ and $\ol{f_2}$ by this homeomorphism produces two maps which we denote $\ol{f_1},\ol{f_2}:\mathscr{U}\to\mathscr{U}$ by a slight abuse of notation.

\begin{definition}[Thurston equivalent graph extensions over $\mathscr{U}$]\label{defn:ThurstonEquivGraphExtensionAnnulus}
We say that two extensions $\ol{f_1},\ol{f_2}:\mathscr{U}\to\mathscr{U}$ of the graph maps $f_1,f_2$  are \emph{Thurston equivalent over $\mathscr{U}$} if there are homeomorphisms $\phi_1, \phi_2:\ol{\mathscr{U}}\to\ol{\mathscr{U}}$ 
so that:
\[
    {\phi_1}\circ \ol{f_1} = \ol{f_2} \circ {\phi_2}
\]
and there exists a homotopy $\Phi: [0,1]\times\overline{\mathscr{U}}\to\overline{\mathscr{U}}$ with
$\Phi(0,\cdot)={\phi_1}$, $\Phi(1,\cdot)={\phi_2}$ so that for all $t\in [0,1]$, we have that $\Phi(t,\cdot)|_{\partial\mathscr{U}}\subset \partial\mathscr{U}$ and $\Phi(t,\cdot)$ restricts to the identity on graph vertices for all $t$.  If ${\phi_1}$ and ${\phi_2}$ are homotopic to the identity in the sense just mentioned, we say that the extensions are homotopic over $\mathscr{U}$.
\end{definition}

Each Hubbard tree $H_i$ is contained in a complementary component $U_i$ of $\Gamma$ so that $U_i\setminus H_i$ is a non-degenerate topological annulus. Let $T_i$ denote the right-hand Dehn twist about this annulus. The  $T_i$ operate on pairwise disjoint annuli because no two $H_i$ lie in the same complementary component of the Newton graph.  Thus any two such twists commute.

Let $R_{1,i}, R_{2,i}$ for $1\leq i\leq r+m$ denote the Newton ray edges connecting $H_i$ to $\Gamma$ in $\Sigma_1,\Sigma_2$ respectively.  By Corollary \ref{cor:CompatibleRayOrbits}, we may assume $R_{1,i}$ and $R_{2,i}$ have the same endpoints and accesses.  The equality symbol is used for arcs to indicate that they are homotopic rel endpoints on $\mathscr{U}$.  Note that for all $i$, there is a unique $\ell_i\in\Z$ and $\ell_i'\in\Z$ so that $T_i^{\ell_i}(R_{1,i})=R_{2,i}$ and $T_{i+1}^{\ell_i'}(\ol{f_1}(R_{1,i}))=\ol{f_2}(R_{2,i})$.


\begin{remark}
It is very likely that the numerical condition in the following lemma can be simplified. For example, in Proposition 2.2 of \cite{BEK}, the Thurston class of a polynomial mating postcomposed by a Dehn twist about the equator is expressed in terms of the same mating with one of the polynomials rotated. A similar sort of behavior is expected when applying a Dehn twist  about a filled Julia set for a Newton map. Though less conceptual, the following lemma can be proven quickly.
\end{remark}

\begin{lemma}[Numerics of equivalent extensions]\label{Lem:NumericsOfAnnulusExtensions}
The extensions $\ol{f_1}$ and $\ol{f_2}$ over $\mathscr{U}$ of the graph maps $f_1$ and $f_2$ are Thurston equivalent if and only if there are integers $n_1,...,n_{r+m-1}$ that satisfy the following system of linear equations:
\begin{equation}\label{eqn:numericsForEquivExtensionsPeriodic}
d_i(n_i-\ell_i)+\ell'_i=n_{i+1}
\end{equation}
where $1\leq i\leq r+m-1$ and $n_r=n_{r+m}$.

\end{lemma}
\begin{proof}

Suppose that the extensions $\ol{f_1}$ and $\ol{f_2}$ are Thurston equivalent.  Then there are $n_i\in\mathbb{Z}$ so that up to branched cover homotopy,
\begin{equation}\label{eqn:ThurstonEquivAlongAnnuli}
S\circ\overline{f_1}= \overline{f_2}\circ S
\end{equation}
where $S=T_1^{n_1}\circ ...\circ T_{m+r-1}^{n_{m+r-1}}$.  

Fix $i$ as in the statement of the lemma.  All of the Dehn twists $T_1,...,T_{m+r-1}$  fix $\overline{f_1}(R_{1,i})$ except possibly $T_{i+1}$, and thus the expression on the left side of Equation \ref{eqn:ThurstonEquivAlongAnnuli} acts on the ray $R_{1,i}$ as follows:
\begin{equation}\label{eqn:ThurstonEquivAlongAnnuliLHS}
S\circ\overline{f_1}(R_{1,i})=  T_{{i+1}}^{n_{i+1}}\circ \overline{f_1}(R_{1,i})
\end{equation}
and the right side acts on $R_{1,i}$ as follows:

\begin{align*}\label{eqn:ThurstonEquivAlongAnnuliRHS}
\overline{f_2}\circ S(R_{1,i}) &=  \overline{f_2}(T_i^{n_i}(R_{1,i})) \\ 
&=  \overline{f_2}(T_i^{n_i-\ell_i}(R_{2,i})) \\ 
&=  T_{i+1}^{d_i(n_i-\ell_i)}(\overline{f_2}(R_{2,i})) \\ 
 &=  T_{i+1}^{d_i(n_i-\ell_i)+\ell'_i}(\overline{f_1}(R_{1,i})).
\end{align*}
Equating the expression in the previous line and the right side of Equation \ref{eqn:ThurstonEquivAlongAnnuliLHS} we obtain Equation \ref{eqn:numericsForEquivExtensionsPeriodic}.

If on the other hand Equation \ref{eqn:numericsForEquivExtensionsPeriodic} has integer solutions, it follows that $S\circ\overline{f_1}$ and $\overline{f_2}\circ S$ are homotopic over $\mathscr{U}$ using a close analog of Lemma \ref{Lem_IsotopGraphMaps}.  Thus the extensions $\ol{f_1}$ and $\ol{f_2}$ are Thurston equivalent over $\mathscr{U}$.
\end{proof}

\begin{remark}
Since the grand orbit of a Hubbard tree can be written as the union of the forward orbit of finitely many Hubbard trees, Definition \ref{defn:ThurstonEquivGraphExtensionAnnulus} extends to the case of complementary components of the Newton graph containing the grand orbit of a Hubbard tree.  Similar numerics as in the previous lemma hold for this slightly more general case.
\end{remark}

\begin{definition}[Newton ray grand orbit equivalence]\label{defn:AENGequivalence}  
\index{grand orbit equivalence!of Newton ray}
We say that the grand orbit of the Newton ray $R_{1,i}\subset\Sigma_1$ landing at $H_i$ is equivalent to the grand orbit of the ray $R_{2,i}\subset\Sigma_2$ landing at $H_{i}$ if after applying Corollary \ref{cor:CompatibleRayOrbits} to guarantee that $R_{1,i}$ and $R_{2,i}$ have the same endpoints and accesses, the numerical condition of Lemma \ref{Lem:NumericsOfAnnulusExtensions} is satisfied.   
\end{definition}

\begin{lemma} \label{Lem_GlobalExtOfGraphEquiv} 
Let $(\Sigma_1,f_1)$ and $(\Sigma_2,f_2)$ be two abstract extended Newton graphs  with $\Sigma_1^-=\Sigma_2^-$ and  $f_1=f_2$. Then if the corresponding Newton ray orbits are equivalent under Definition \ref{defn:AENGequivalence}, the extensions of $f_1$ and $f_2$ to the sphere are Thurston equivalent as marked covers.
\end{lemma}

\begin{proof}
Applying Lemma \ref{cor:CompatibleRayOrbits} to $f_2:\Sigma_2\to\Sigma_2$, we may assume that the Newton ray grand orbit of $f_2$ has the same endpoints and accesses as $\Sigma_1$.   Let  $\mathscr{U}$ be the union of the complementary components of $\Sigma_2^-$ that contain a non-degenerate Hubbard tree in the closure. Then define the homeomorphisms $\phi_1,\phi_2:\S^2\to\S^2$  to be the identity on $\S^2\setminus \mathscr{U}$ and on the components of $\mathscr{U}$ (which are necessarily annuli), the maps are defined to be the self-homeomorphism of $\mathscr{U}$ from Lemma \ref{Lem:NumericsOfAnnulusExtensions}. Then Lemma \ref{Lem:NumericsOfAnnulusExtensions} and Lemma \ref{Lem_IsotopGraphMaps} imply that $f_1$ and $f_2$ are Thurston equivalent as marked covers.
\end{proof}

\subsection{Equivalence on abstract extended Newton graphs}\label{subsec_equivOnAENG}

We now define the combinatorial equivalence relation on abstract extended Newton graphs that is used in the classification theorem (Theorem \ref{Thm_Bijection}) and prove an important result connecting this equivalence with Thurston equivalence. Note that the simplifying assumptions of Remark \ref{rem:raySimplifying} are in effect for Section \ref{subsec_equivOnAENG}.

\begin{lemma}[Extension across Newton rays] \label{Lemma_ThurstonEquivalenceAbstExtGraphsPrep} 
Let $(\Sigma_1,f_1)$ and $(\Sigma_2,f_2)$ be abstract extended Newton graphs. Let $\phi_1^-, \, \phi_2^-\colon \Sigma_1^-  \to \Sigma_2^- $ be graph homeomorphisms that preserve the cyclic
order of edges at all the vertices of $\Sigma_1^-, \, \Sigma_2^-$, and satisfy the equation $\phi_1^-\circ f_1=f_2\circ \phi_2^-$ on $\Sigma_1^-$. Then if the accesses/endpoints of each Newton ray in $\Sigma_1$  correspond under $\phi_1^-$ and $\phi_2^-$ to the accesses/endpoints of a Newton ray in $\Sigma_2$, then $\phi_1^-,\phi_2^-$ extend to graph homeomorphisms $\phi_1,\phi_2:\Sigma_1\to\Sigma_2$ that preserve the cyclic order of edges at each vertex, with $\phi_1\circ f_1=f_2\circ \phi_2$ on $\Sigma_1$.
\end{lemma}

\begin{proof}
For a given Newton ray $R$ in $\Sigma$, the graph maps $\phi_1$ and $\phi_2$ are already defined at the endpoints $i(R)$ and $t(R)$ of $R$. Thus the image of the single edge $R$ under $\phi_1$ must be taken to be the unique Newton ray in $\Sigma_2$ connecting $\phi_1(t(R))$ and $\phi_1(i(R))$, and likewise for $\phi_2$. The extension of the conjugacy across $R$ is accomplished by pullback.
\end{proof}

\begin{definition} [Equivalence relation for abstract extended Newton graphs] \label{Def_ThurstonEquivalenceAbstExtGraphs}
\index{equivalent!abstract extended Newton graphs}\index{abstract extended Newton graph!equivalence} 
Let $(\Sigma_1,f_1)$ and $(\Sigma_2,f_2)$ be abstract extended Newton graphs. We say that $(\Sigma_1,f_1)$ and $(\Sigma_2,f_2)$ are
\emph{equivalent} if and only if
\begin{itemize}
\item there exist two homeomorphisms
$\phi_1^-, \, \phi_2^-\colon \Sigma_1^-  \to \Sigma_2^- $ that are graph maps and preserve the cyclic
order of edges at all the vertices of $\Sigma_1^-, \, \Sigma_2^-$,
\item the equation $\phi_1^-\circ f_1=f_2\circ \phi_2^-$ holds on $\Sigma_1^-$, and  
\item upon applying Corollary \ref{cor:CompatibleRayOrbits} and Lemma \ref{Lemma_ThurstonEquivalenceAbstExtGraphsPrep} to produce the extensions $\phi_1,\phi_2:\Sigma_1\to\Sigma_2$,   each ray grand orbit of $f_2$ is equivalent (see Definition \ref{defn:AENGequivalence}) to a ray grand orbit of $\phi_1\circ f_1\circ \phi_2^{-1}$ and vice versa.
\end{itemize}
\end{definition}

We now show that two abstract extended Newton graphs are combinatorially equivalent if and only if their extensions are Thurston equivalent.

\begin{theorem}[Combinatorial formulation of Thurston equivalence]\label{Thm_CombinatorialFormOfThurstonEquiv}
\index{Thurston equivalent!combinatorial formulation}
Let $(\Sigma_1,f_1)$ and $(\Sigma_2,f_2)$ be abstract extended Newton graphs with graph map extensions $\ol{f_1},\ol{f_2}:\S^2\to\S^2$ respectively.  Then $(\Sigma_1,f_1)$ and $(\Sigma_2,f_2)$ are equivalent in the sense of Definition \ref{Def_ThurstonEquivalenceAbstExtGraphs} if and only if $(\ol{f_1},\Sigma_1')$ and $(\ol{f_2},\Sigma_2')$ are Thurston equivalent as marked branched covers.
\end{theorem}

\begin{proof}

First assume that the two abstract extended Newton graphs are equivalent.  By definition of extended Newton graph equivalence, there are graph homeomorphisms $\phi_1, \, \phi_2:\Sigma_1 \to \Sigma_2$ that satisfy the conditions in Definition \ref{Def_ThurstonEquivalenceAbstExtGraphs}. Since the complementary components of $\Sigma_1$ and $\Sigma_2$ are all disks, $\phi_1,\phi_2$ can be extended to global homeomorphisms $\bar{\phi_1},\bar{\phi_2}$. Since the Newton ray grand orbits are equivalent in the sense of Definition \ref{defn:AENGequivalence}, Lemma \ref{Lem_GlobalExtOfGraphEquiv} implies that there must be homeomorphisms $S_1$ and $S_2$ of the sphere which are both products of Dehn twists about the non-degenerate Hubbard trees as in Lemma \ref{Lem:NumericsOfAnnulusExtensions}  so that
\begin{equation}\label{Eqn_MultitwistCommuteWithNewton}
S_1\circ\bar{\phi}_1\circ{f_1}=f_2\circ\bar{\phi}_2\circ S_2
\end{equation}
where $S_1$ is homotopic to $S_2$ relative to the vertices of $\Sigma_1$.  The maps on both sides of Equation \ref{Eqn_MultitwistCommuteWithNewton} are both regular extensions of a graph map (see Proposition \ref{Prop_RegExt}) and they also satisfy the hypotheses of Lemma \ref{Lem_IsotopGraphMaps}. Thus $f_1$ and $f_2$ are Thurston equivalent as marked branched covers.

Now suppose that $(\overline{f}_1,\Sigma'_1)$ and $(\overline{f}_2,\Sigma'_2)$ are Thurston equivalent as marked branched covers.  Take $g_0,g_1:(\S^2, \Sigma'_1)\to (\S^2,\Sigma'_2)$ to be the maps from the definition of Thurston equivalence where $g_0\circ\ol{f_1}=\ol{f_2}\circ g_1$.  Let $e$ be an edge of $\Sigma_1^-$ with endpoints $\partial e$.  Then $g_1(e)$ connects the two points in $g_1(\partial e)$.
Moreover, $g_1$ preserves the cyclic order at each vertex of
$\Sigma_1^-$, because it is an orientation-preserving homeomorphism of $\S^2$. Let
$g':(\S^2,\Sigma'_2)\to(\S^2,\Sigma'_2)$ be a homeomorphism that for all edges $e$  not Newton rays maps each $g_1(e)$
to an edge of $\Sigma_2$ that connects the two points in $g_1(\partial e)$. 

Then $g'\circ g_1$ realizes an equivalence between the two abstract extended Newton graphs (let $\phi_0=\phi_1=g'\circ g_1$ in Definition \ref{Def_ThurstonEquivalenceAbstExtGraphs}), except that the Newton rays must still be shown to be equivalent.  Apply Corollary \ref{cor:CompatibleRayOrbits} so that all Newton rays have corresponding endpoints under $\phi_0,\phi_1$.  Then since $\ol{f_1}$ and $\ol{f_2}$ are Thurston equivalent as branched covers, Lemma \ref{Lem:NumericsOfAnnulusExtensions} implies the rays are equivalent.  Thus $(\Sigma_1,f_1)$ and $(\Sigma_2,f_2)$ are combinatorially equivalent.
\end{proof}


\section{Newton maps from abstract extended Newton graphs}\label{NewtonMapsfromGraphs}

We now prove that all abstract extended Newton graphs are realized by Newton maps.

\begin{proof}[Proof of Theorem \ref{Thm_Realization}]
It suffices to show that the marked branched cover $(\ol{f},\Sigma')$ is unobstructed, where $\Sigma'$ denotes the set of vertices of $\Sigma$.  The conclusion will follow by Head's theorem, where the holomorphic fixed point theorem is used to argue that the point at infinity is repelling \cite{Milnorbook}.

Suppose to the contrary that $\Pi$ is a multicurve obstruction for $(\ol{f},\Sigma')$, and without loss of generality assume $\Pi$ is irreducible.  Recall from Condition \ref{Cond_NewtonGrah} that $\Sigma$ contains an abstract Newton graph $\Gamma$ which in turn contains an abstract channel diagram $\Delta$.   The following lemma restricts where obstructions may exist, using Theorem \ref{Thm_ArcsInterOb}.

\begin{lemma}\label{Lemma:ObstructionsIntersectingEdges}
If $\Pi$ is a multicurve obstruction\index{multicurve obstruction}\index{obstruction!multicurve} for $(\ol{f},\Sigma')$, then $$\Pi\cdot(\Gamma\setminus\Delta)=0.$$
\end{lemma}

\begin{proof}
Suppose first that there exists an edge $\lambda$ in $\Delta$ so that $\lambda\cdot\Pi\neq 0$.  Since $\{\lambda\}$ itself forms an irreducible arc system, the second case of Theorem \ref{Thm_ArcsInterOb} implies that  $\Pi$ intersects no other preimage of $\lambda$ except for $\lambda$ itself.  

If on the other hand, $\lambda\cdot\Pi = 0$, the first case of Theorem \ref{Thm_ArcsInterOb} implies that no preimage of $\lambda$ intersects $\Pi$.  Since every edge in $\Gamma$ is a lift of an edge in $\Delta$, the conclusion follows.
\end{proof}

The proof of the theorem is now completed by showing that whether or not $\Pi$ has intersection with $\Delta$ in minimal position, a contradiction results.

\subsection{Contradiction for the case $\Pi\cdot\Delta\neq 0$}

Let $\gamma_1$ be any curve in $\Pi$ so that $\gamma_1\cdot\Delta\neq 0$.  Recall from Definition \ref{Def_NewtonGraph} that $\overline{\Gamma\setminus\Delta}$ is connected. It is a consequence of Lemma \ref{Lemma:ObstructionsIntersectingEdges} that $\overline{\Gamma\setminus\Delta}$ does not intersect $\Pi$; thus $\overline{\Gamma\setminus\Delta}$ must be a subset of one of the complementary components of $\gamma_1$.  Denote the complementary component of $\gamma_1$ that does not contain $\overline{\Gamma\setminus\Delta}$ by $D(\gamma_1)$.  None of the vertices of $\Gamma$ except possibly $v_\infty$ lie in $D(\gamma_1)$.  However, there must be at least two vertices of $\Sigma$ in $D(\gamma_1)$ for otherwise $\gamma_1$ would not be essential.  The only vertices of $\Sigma$ which could possibly be in $D(\gamma_1)$ are $v_\infty$ and Hubbard tree vertices. Due to the connectedness of $\Sigma$, at least one of the following must hold: $\gamma_1\cdot H_1\neq 0$ for some Hubbard tree $H_1$ or $\gamma_1\cdot R_1\neq 0$ for some Newton ray $R_1$. We only prove the Hubbard tree case, noting that the Newton ray case is identical.


\begin{figure}[h]
\centerline{\includegraphics[width=100mm]{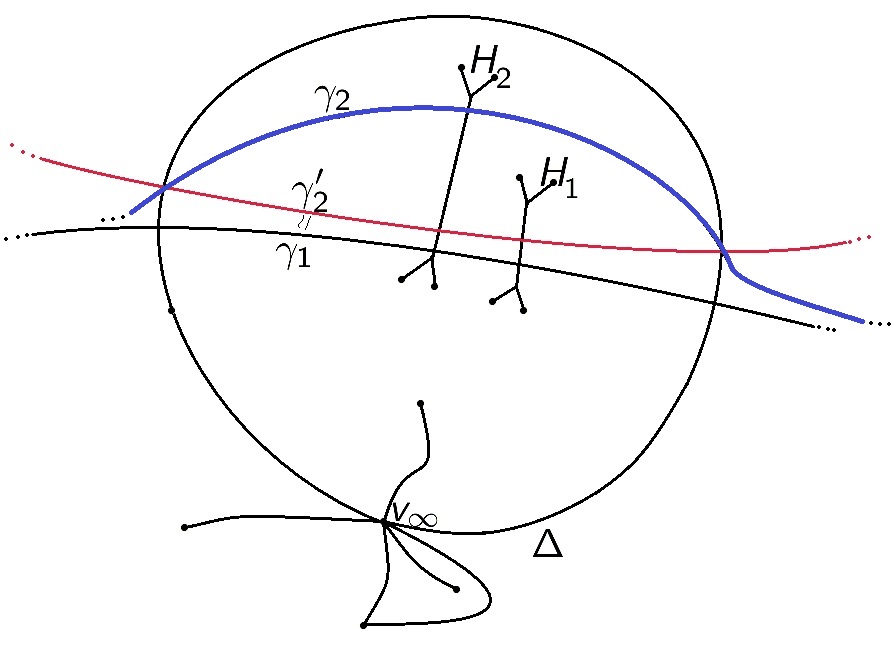}}
\caption{Illustration of the first case $\Pi\cdot\Delta\neq 0$} 
\label{Fig_Deg4SmallBasilicas}
\end{figure}

First suppose $\gamma_1\cdot H_1\neq 0$. Let $\gamma_2\in\Pi$ be some curve whose preimage under $\ol{f}$ has a component $\gamma_2'$ which is homotopic to $\gamma_1$ rel vertices ($\gamma_2$ exists by irreducibility).  Clearly $\gamma_2'$ must intersect $H_1$ and $\Delta$.  Let $\lambda_1$ be some component of $\gamma_2'\cap(S^2\setminus\Delta)$ that intersects $H_1$.  Recall that $H_1$ must either be non periodic or have period at least two, so in either case $H_2:=\ol{f}(H_1)$ does not intersect $H_1$.  Also $H_1$ and $H_2$ must be in the same complementary component of $\Delta$ because $\lambda_1$ connects $H_1$ to $\Delta$ (without passing through any other edges of $\Gamma$ due to Lemma \ref{Lemma:ObstructionsIntersectingEdges}), and $\ol{f}$ is an orientation-preserving map that fixes each edge of $\Delta$.

Now we show that $H_1$ and $H_2$ can be connected by some path that avoids $\Gamma$ except at its endpoints. Since all critical points of $\ol{f}$ are contained in $\Sigma$ (by Condition \ref{Cond_DegreeExtGraph} of Definition \ref{Def_AbstractExtNewtGraph}) the Riemann--Hurwitz formula implies that $\ol{f}^{-1}(\Delta)\subset \Gamma$.  The edges of $\Delta$  are invariant under $\ol{f}$, so the endpoints of $\lambda_2:=\ol{f}(\lambda_1)$ are in the same edges as the endpoints of $\lambda_1$.  Since $\ol{f}^{-1}(\Delta)\subset\Gamma$, we have that $\lambda_2$ intersects $\Delta$ only at its endpoints. Starting at an intersection of $H_1$ and $\lambda_1$, traverse $\lambda_1$ until right before the intersection with the edge of $\Delta$.  Traverse a path in a small neighborhood of this edge until $\lambda_2$ is reached without intersecting any edges of $\Gamma$.  Traverse $\lambda_2$ until $H_2$ is reached.  This completes the construction of a path $\lambda_{1,2}$ from $H_1$ to $H_2$ that does not intersect $\Delta$. Moreover, Lemma \ref{Lemma:ObstructionsIntersectingEdges} implies that $\lambda_1$ and $\lambda_2$ do not intersect $\Gamma\setminus\Delta$ avoiding $\Gamma$ and so $\lambda_{1,2}$ does not intersect $\Gamma$. This contradicts the assumption that $H_1$ and $H_2$ are separated by the Newton graph (Condition~\ref{Cond_TreesSeparated}).

\subsection{Contradiction for the case $\Pi\cdot\Delta= 0$}

Using Lemma \ref{Lemma:ObstructionsIntersectingEdges} we see that $\Pi\cdot\Gamma=0$.  Recall the assumption that every complementary component of $\Gamma$ contains at most one abstract extended Hubbard tree (Condition \ref{Cond_TreesSeparated}).  

Suppose that $U$ is such a complementary component containing some $\gamma\in\Pi$.  The only postcritical points that could possibly be contained in $U$ are vertices of Hubbard trees, so $U$ contains one Hubbard tree or one Hubbard tree preimage.  Since $\Pi$ is irreducible, the Hubbard tree must in fact be periodic and since $\gamma$ is essential the Hubbard tree is non-degenerate. Thus $U$ contains exactly one non-degenerate periodic abstract Hubbard tree $H$ of some period $m$.  Define $F:=\ol{f}^m$, and note that $\Pi$ is also a multicurve obstruction for $F$. Extract an irreducible multicurve obstruction for $F$ from $\Pi$, which we again denote by $\Pi$, and assume that $U$ still contains some component of $\Pi$.  

We show that the two Thurston linear maps $F_{\Pi}$ and $(F|_U)_{\Pi}$ are equal. In fact, we show $\Pi\subset U$.  Suppose that $W$ is a complementary component of $\Gamma$ different from $U$, and $\gamma'\subset W$ for some $\gamma'\in\Pi$.  By the irreducibility of $\Pi$, there is some $n>0$ and a component $\gamma''$ of $F^{-n}(\gamma')$ that is homotopic to $\gamma$ rel vertices.  Note that $\gamma''\subset U$ and that its complementary component that is a subset of $U$ contains some vertices of $\Sigma$ which must in fact be vertices of $H$. Since $\gamma''$ is homotopic to a subset of each arbitrarily small neighborhood of $H$, we obtain a contradiction since $F^n(\gamma'')\subset W$ but $F^n(H)=H\subset U$. Thus the two Thurston linear maps $F_{\Pi}$ and $(F|_U)_{\Pi}$ are equal. 

This contradicts the realizability (or unobstructedness) of the abstract Hubbard tree $H$ \cite[Theorem II.4.7]{Poirier}, and thus no such obstruction $\Pi$ exists, completing the proof.

\end{proof}


\section{Proof of the classification theorem} \label{Sec_ProofThmBijection}

Theorem 1.2 of \cite{LMS1} asserts that every postcritically finite Newton map has an extended Newton graph that satisfies the axioms of Definition \ref{Def_AbstractExtNewtGraph}, and we have shown in Section \ref{NewtonMapsfromGraphs} that every abstract extended Newton graph extends to an unobstructed branched cover, and is therefore realized by a Newton map.  We now check that these two assignments are well-defined on equivalence relations and are inverses of each other, giving an explicit classification of postcritically finite Newton maps in terms of combinatorics.

Recall that $\Newt$ is the set of postcritically finite Newton maps up to affine conjugacy, and that $\AENG$ is the set of abstract extended Newton graphs up to Thurston equivalence (Definition \ref{Def_ThurstonEquivalenceAbstExtGraphs}).  Equivalence classes in both cases are denoted by square brackets.  Our first goal is to show that the assignments made in Theorems \ref{Thm_NewtMapsGenerateExtNewtGraphs} and \ref{Thm_Realization} are well-defined on the level of equivalence classes, namely, they induce mappings $\mathcal{F}: \Newt \to \AENG$ and $\mathcal{F}': \AENG \to \Newt$.

We now argue that $\mathcal{F}$ is well-defined.  The construction from \cite{LMS1} of the extended Newton graph for a fixed Newton map involved no choices in the construction of type H and N edges, and possibly many choices in the construction of type R edges.  Let $(\Delta_{\N,1}^*,N_p)$ and $(\Delta_{\N,2}^*,N_p)$ be two extended Newton graphs constructed for $N_p$.  Proposition 6.4 in \cite{LMS1} asserts that $\Delta_{\N,1}^-=\Delta_{\N,2}^-$ and $N_p|_{\Delta_{\N,1}^-}=N_p|_{\Delta_{\N,2}^-}$ (recall that $\Delta_{\N,1}^-$ denotes the graph $\Delta_{\N,1}$ with all Newton ray edges removed).  We thus only need to show that the Newton ray grand orbits are equivalent.  The branched cover $(N_p,(\Delta_{\N,1}^*)')$ is identical as a branched cover to $(N_p,(\Delta_{\N,2}^*)')$ and they are both extensions of graph maps $N_p|_{\Delta_{\N,1}^*}$ and $N_p|_{\Delta_{\N,2}^*}$ respectively.  Theorem \ref{Thm_CombinatorialFormOfThurstonEquiv} then implies equivalence for corresponding ray grand orbits.

Well-definedness of $\mathcal{F}'$ is immediate from the fact that equivalent graphs have Thurston equivalent extensions (Theorem \ref{Thm_CombinatorialFormOfThurstonEquiv}) which correspond to affine conjugate Newton maps by Thurston rigidity (Theorem \ref{Thm_Thurston'sTheorem}).

\begin{proof}[Proof of Theorem \ref{Thm_Bijection}]
We first show injectivity of $\mathcal{F}: \Newt \to \AENG$.
Let $N_{p_1}$ and $N_{p_2}$ be two postcritically finite Newton maps that have equivalent extended Newton graphs $\Delta^*_{\N,1}$ and $\Delta^*_{\N,2}$.  Theorem 
\ref{Thm_NewtMapsGenerateExtNewtGraphs} asserts that each of these graphs satisfies the axioms of an abstract extended Newton graph, and since both graphs are equivalent, the marked branched covers $(N_{p_1},(\Delta^*_{\N,1})')$ and $(N_{p_2},(\Delta^*_{\N,2})')$ are equivalent by Theorem \ref{Thm_CombinatorialFormOfThurstonEquiv}.   We may then conclude that $N_{p_1}$
and $N_{p_2}$ are affine conjugate using Thurston rigidity.

Next we show injectivity of $\mathcal{F}': \AENG \to \Newt$.  Suppose that a postcritically finite Newton map $N_p$ realizes two abstract extended Newton graphs $(\Sigma_1,f_1)$ and $(\Sigma_2,f_2)$.  By minimality of the extended Hubbard trees and the Newton graph, we know that $\Sigma_1'=\Sigma_2'$.  Then the marked branched covers $(N_p,\Sigma_1')$ and $(N_p,\Sigma_2')$ are Thurston equivalent.  By Theorem \ref{Thm_CombinatorialFormOfThurstonEquiv} we conclude that $(\Sigma_1,f_1)$ and $(\Sigma_2,f_2)$ are combinatorially equivalent.

Finally we prove that $\mathcal{F}$ and $\mathcal{F}'$ are bijective and inverses of each other. Let $(\Sigma,f) \in \AENG$ be an abstract extended Newton graph. It follows from Theorem \ref{Thm_Realization} that $(\Sigma,f)$ is realized by a postcritically finite Newton map $N_p$. Thus 
\[
\mathcal{F}'([(\Sigma,f)]) = [N_p].
\]
Denote by $\Delta^*_\N$ an extended Newton graph associated with $N_p$ which is guaranteed by Theorem \ref{Thm_NewtMapsGenerateExtNewtGraphs} so that
\[
\mathcal{F}([N_p]) =\left[\left(\Delta^*_\N, N_p\right) \right].
\]
The injectivity statement just proved implies that under the equivalence of Definition \ref{Def_ThurstonEquivalenceAbstExtGraphs}, 
\[
[(\Sigma,f)] = \left[\left(\Delta^*_\N,N_p\right) \right].
\]
Thus $\mathcal{F} \circ \mathcal{F}'$ is the identity, and consequently the mapping $\mathcal{F}: \Newt \to \AENG$ is bijective and $\mathcal{F}' \circ \mathcal{F}$ is the identity. 
\end{proof}

This completes the combinatorial classification of postcritically finite Newton maps.

\bibliography{Lodge}
\bibliographystyle{alpha}

\printindex

\end{document}